\documentclass[12pt,leqno]{article}
\usepackage{amssymb,amsmath,amsthm}
\usepackage[all]{xy}
\usepackage[alphabetic,lite]{amsrefs}
\numberwithin{equation}{section}

\usepackage[top=2cm,bottom=2cm,left=2cm,right=4cm,marginparsep=0.3cm,marginparwidth=3cm,includefoot]{geometry}
\usepackage{mathtools}
\usepackage{unusual}
\usepackage{mathrsfs}
\usepackage{xcolor}

\usepackage[colorlinks=true, pdfstartview=FitV, linkcolor=blue, citecolor=blue, urlcolor=blue]{hyperref}

\begin{document}

\title
{Unusual functorialities for weakly constructible sheaves}
\date{}

\author{Andreas Hohl and Pierre Schapira}
\maketitle

\abstract{We prove that various morphisms related to the six Grothendieck operations on sheaves become isomorphisms 
	when restricted to (weakly) constructible sheaves. To this end, we first study some properties of weakly cohomologically constructible sheaves. We then deduce several compatibilities of the six operations in the context of (weakly) $\R$-constructible sheaves.}

{\renewcommand{\thefootnote}{\mbox{}}
\footnote{The research of A.H. was funded  by the Deutsche Forschungsgemeinschaft (DFG, German Research
	Foundation), Projektnummer 465657531.}
\footnote{\emph{Key words.} constructible sheaves, six-functor formalism, ind-objects and pro-objects, subanalytic sets.}
\footnote{2020 \emph{Mathematics Subject Classification.} 32S60, 18G80, 32B20.}
\addtocounter{footnote}{-2}
}

\section{Introduction}
Let $\cor$ be a commutative unital ring of finite global dimension and 
denote by $\Derb(\cor_X)$ the bounded derived category of sheaves on a good topological space $X$.
There are some classical morphisms which are, in general, not isomorphisms, such as
\eqn
&&\rhom(F_1,F_2)\ltens F_3\to\rhom(F_1,F_2\ltens F_3),\\
&&\epb{f}G_1\ltens\opb{f}G_2\to\epb{f}(G_1\ltens G_2),
\eneqn
for $F_1,F_2,F_3\in\Derb(\cor_X)$, $G_1,G_2\in\Derb(\cor_Y)$ and $f\cl X\to Y$  a continuous map.
We will prove here that, under suitable hypotheses of (weak) constructibility, these morphisms become isomorphisms.

We prove the following results in the categories of real analytic manifolds and weakly $\R$-constructible sheaves (see Theorems~\ref{th:invim},~\ref{th:tenshom} and~\ref{th:dirim}). 

Let $f\cl X\to Y$ be a morphism of real analytic manifolds, and let $F_1, F_2, K\in  \Derb_\wRc(\cor_X)$,  $G, M\in  \Derb_\wRc(\cor_Y)$
with moreover $F_1$ being $\R$-constructible and $K,M$ locally constant.
Then we have the isomorphisms
\eqn
&& \epb{f}G\tens\opb{f}M\isoto\epb{f}(G\tens M),\\
&&\opb{f}\rhom(M,G)\isoto \rhom(\opb{f}M,\opb{f}G),\\
&&\rhom(F_1,F_2)\tens K\isoto\rhom(F_1,F_2\tens K),\\
&&\rhom(K,F_2)\tens F_1\isoto\rhom(K,F_2\tens F_1).
\eneqn
For direct images, we need slightly stronger assumptions: Consider  a morphism $f\cl \fX\to\fY$  of b-analytic manifolds (see~\cite{Sc23} for this notion) and let   $F\in \Derb_\wRc(\cor_\fX)$ be weakly $\R$-constructible up to infinity and $M$ be locally constant. We prove the isomorphisms 
\eqn
&&\roim{f}F\tens M\isoto \roim{f}(F\tens \opb{f}M),\\
&&\reim{f}\rhom(\opb{f}M,F)\isoto \rhom(M,\reim{f}F). 
\eneqn

We start by introducing the notion (implicitly already defined in~\cite[\S3.4]{KS90}) of weakly  cohomologically constructible sheaves and the full subcategory $\Derb_\wcc(\cor_X)$ of $\Derb(\cor_X)$ consisting of such objects. On a  real analytic manifold, the category $\Derb_\wcc(\cor_X)$   contains the category $\Derb_\wRc(\cor_X)$ of weakly $\R$-constructible objects.

We prove first that  $\Derb_\wcc(\cor_X)$ is triangulated. Then our main tool is that for an object $F$ of this category, for $x\in X$ and $L\in\Derb(\cor)$, one has functorial isomorphisms
\eqn
&&\rhom(L_X,F)_x\isoto \rhom(L,F_x),\quad \rsect_xF\tens L\isoto \rsect_x(F\tens L_X),
\eneqn
where $L_X$ denotes the constant sheaf associated with $L$.

The motivation for this note came through the work \cite{Ho23}, where field extensions for sheaves are considered: In this context, given a field extension $\cor \subset \corex$, there are natural functors of extension and co-extension of scalars from $\Derb(\cor_X)$ to $\Derb(\corex_X)$, which are given by $F\mapsto F\otimes \corex_X$ and $F\mapsto \rhom(\corex_X,F)$, respectively. It is important to study the question of compatibility of these functors with the six Grothendieck operations on constructible sheaves. It turns out that many of the desired functorialities of loc.~cit.\ indeed follow from the set-up developed here,  but our approach is more general in multiple aspects: We allow $\cor$ to be a (suitable) commutative ring, we only require some of the sheaves to be \emph{weakly} constructible, and we replace $\corex_X$ by any locally constant sheaf.

\paragraph{Acknowledgements} The authors warmly thank one of the referees who made the essential remark that a hypothesis asserting that some duality functor was conservative was not necessary in order to prove 
Proposition~\ref{pro:dualconserv} and, as a byproduct, Theorems~\ref{th:invim},~\ref{th:tenshom} and~\ref{th:dirim}. 

\section{Preliminaries}
In all this paper, we work in a given universe $\mathcal{U}$. All limits and colimits (in particular, products and direct sums) are assumed to be small.
Recall that $\cor$ is  a commutative unital ring of finite global dimension. 
We assume that  all topological spaces are ``good'', that is, Hausdorff, locally compact, countable at infinity and of finite flabby dimension.

For a topological space $X$ as above, 
one denotes by $\md[\cor_X]$ the Grothendieck abelian category of sheaves of $\cor$-modules and by  $\Derb(\cor_X)$ its bounded derived category.

We mainly follow the notations of~\cite{KS90}. In particular, 

\begin{itemize}
	\item
	$\omega_X$ denotes the dualizing complex and $\RD_X$  the duality functor 
	$\rhom(\scbul,\omega_X)$,
	\item
	$\RD$ denotes the duality functor on $\Derb(\cor)$,
	\item
	$a_X\cl X\to \{\rmpt\}$ denotes the unique map from $X$ to a one-point space. Hence, for 
	$F\in\Derb(\cor_X)$, one has $\rsect(X;F)\simeq\roim{a_X}F$.
	\item
	For $L\in\Derb(\cor)$,
	and for $Z$ locally closed in $X$, one denotes by $L_{XZ}$ the  constant sheaf $L_Z$ on $Z$ extended by $0$ on $X\setminus Z$. When $Z$ is closed, we shall often simply denote by $L_Z$ the sheaf $L_{XZ}$.
	\item
	For $x\in X$, denoting by $i_x\cl \{x\}\into X$ the embedding, and for $F\in\Derb(\cor_X)$, one denotes as usual by $F_x=\opb{i_x}F$ its stalk at $x$. One also sets $\rsect_xF=\epb{i_x}F$. (We sometimes identify $\opb{i_x}F$ and $\oim{i_x}\opb{i_x}F$ as well as $\epb{i_x}F$ and $\oim{i_x}\rsect_xF$.)
\item
We often denote by $K$ a locally constant sheaf on $X$ and by $M$ a locally constant sheaf on $Y$. 
As already mentioned, if $L\in\Derb(\cor)$, we denote by $L_X$ or $L_Y$ the associated constant sheaf on $X$ or $Y$. 
\item
Recall (see~\cite[Exe.~I.30]{KS90}) that $L\in\Derb(\cor)$ is {\em perfect} if it is isomorphic to a bounded complex of finitely generated projective $\cor$-modules. If $L$ is perfect, then so is $\RD(L)$ and the morphism $L\to\RD\RD(L)$ is an isomorphism.
	We shall denote  by $\Derb_f(\cor)$ the full triangulated category of $\Derb(\cor)$ consisting of perfect objects.
\end{itemize}

\subsubsection*{Ind-objects}
We shall make use of ind-objects. For a short  exposition see~\cite[\S1.11]{KS90}. For a more detailed
study, including new results that we shall use here, see~\cite[Ch.~6, \S8.6, Ch.~15]{KS06}. Let us recall a few facts that we need, skipping some delicate questions of universes. 

If $\shc$ is a category, one denotes by $\IC$ the category of ind-objects of $\shc$, a full subcategory of 
the category $\shc^\wedge$ of functors from $\shc^\rop$ to $\Set$. 
Recall from \cite[\S6.1]{KS06}: 
\begin{itemize}
	\item
	The natural functor $\shc\to\IC$ is fully faithful.
	\item
	The category $\IC$ admits small filtrant colimits, denoted $\sinddlim$. 
	\item
	Let $\shi$ be a small and filtrant category and  $\alpha\cl \shi\to\shc$ a functor.
	Let $T\cl\shc\to\shc'$ be a functor. Then 
	$T(\sinddlim\alpha)\simeq\sinddlim (T\circ\alpha)$. 
\end{itemize}

Now we assume that $\shc$ is abelian. Recall from {\cite[Th.~8.6.5]{KS06}  that:
	\begin{itemize}
		\item
		The category $\IC$ is abelian and the fully faithful  functor $\shc\to\IC$ is exact.
		\item
		Small filtrant colimits  are exact in $\IC$.
	\end{itemize}
	One should be aware that even if $\shc$ is a Grothendieck category, $\IC$ does not admit enough injectives in general.
	
	Let $\shi$ be a small category and $\alpha\cl \shi\to \shc$ be a functor. As already mentioned, one denotes by $\sinddlim\alpha$ its colimit in $\IC$.
	Note that if $\shc$ admits colimits, denoted $\sindlim$,   there is a natural morphism in $\IC$
	\eqn
	&&\sinddlim\alpha\to\sindlim\alpha
	\eneqn
	but this morphism is not an isomorphism in general. However:
	\begin{itemize}
		\item
		If $\sinddlim\alpha$ belongs to $\shc$, then $\sinddlim\alpha\isoto\sindlim\alpha$
		(see~\cite[Cor.~1.11.7]{KS90}). In this case, if $T\cl\shc\to\shc'$ is a functor, then
		$\sinddlim (T\circ\alpha)\isoto T(\sindlim\alpha)$. Therefore, $\sinddlim (T\circ\alpha)$ belongs to $\shc'$  and hence 
		is isomorphic to $\sindlim (T\circ\alpha)$.
	\end{itemize} 
	
	\section{Weakly cohomologically constructible sheaves}
	
	We consider here a slight generalisation of the notion of cohomologically constructible sheaves (see~\cite[Def.~3.4.1]{KS90}).
	
	\begin{definition}\label{def:wcc}
		Let $F\in\Derb(\cor_X)$. We say that $F$ is  {\em weakly cohomologically constructible} if for all $x\in X$, one has the isomorphisms
		\eqn
		&&\inddlim[x\in U]\rsect(U;F)\isoto F_x,\quad \rsect_xF\isoto\proolim[x\in U]\rsect_c(U;F).
		\eneqn
		We denote by  $\Derb_\wcc(\cor_X)$ the full subcategory of  $\Derb(\cor_X)$ consisting of 
		weakly cohomologically constructible objects.
	\end{definition}
	Recall that $F$ is cohomologically constructible if $F$ is weakly cohomologically constructible and moreover $F_x$ and $\rsect_xF$ are perfect objects of 
	$\Derb(\cor)$ for all $x\in X$. One denotes by $\Derb_\cc(\cor_X)$ the full subcategory of  $\Derb_\wcc(\cor_X)$ consisting of cohomologically constructible objects.

	\begin{remark}
		As explained in~\cite[Rem.~4.3.2]{KS90}, the isomorphisms in Definition~\ref{def:wcc} hold as soon as 
		the objects $\inddlim[x\in U]\rsect(U;F)$ and $\proolim[x\in U]\rsect_c(U;F)$ are representable.
	\end{remark}
	
	\begin{proposition}\label{pro:tri}
		The category $\Derb_\wcc(\cor_X)$ is triangulated. 
	\end{proposition}
	\begin{proof}
		(i) Remark first that for $F\in  \Derb_\wcc(\cor_X)$, $x\in X$  and $j\in\Z$, one has
		\eqn
		&&\inddlim[x\in U]H^j(U;F)\isoto H^j(F)_x.
		\eneqn
		(ii) Clearly, if $F\in  \Derb_\wcc(\cor_X)$, then so is the shifted sheaf $F[j]$ for $j\in\Z$.
		
		\spa
		(iii) Consider a distinguished triangle $F'\to F\to F''\to[+1]$ in $\Derb(\cor_X)$ and assume that 
		$F',F''\in \Derb_\wcc(\cor_X)$. Let $x\in X$ and let $U$ be an open neighborhood of $x$.
		We get the morphism of long exact sequences in the abelian category $\md[\cor]$
		\eq\label{eq:lesxU}
		&&\ba{l}\xymatrix{
			{\cdots}\ar[r]& H^j(U;F')\ar[r]\ar[d]& H^j(U;F)\ar[r]\ar[d]&H^j(U;F'')\ar[r]\ar[d]&H^{j+1}(U;F')\ar[r]\ar[d]& \cdots\\
			{\cdots}\ar[r]& H^j(F')_x\ar[r]        & H^j(F)_x\ar[r]        &H^j(F'')_x\ar[r]&H^{j+1}(F')_x\ar[r]       & \cdots.
		}\ea
		\eneq
		Applying the functor $ \inddlim[x\in U]$ and using~\cite[Th.~8.6.5]{KS06}, the first line gives rise to   the long exact sequence  
		in the abelian category $\Ind(\md[\cor])$:
		\eqn
		&&\cdots\to H^j(F')_x\to \inddlim[x\in U]H^j(U;F)\to H^j(F'')_x\to H^{j+1}(F')_x\to\cdots
		\eneqn
		We shall apply~\cite[Lemma~15.4.6]{KS06}, following  its notations, to the category $\shc=\md[\cor]$. 
		Consider the morphism
		\eqn
		&&\phi\cl \inddlim[x\in U]\rsect(U;F)\to F_x.
		\eneqn
		It follows from~\eqref{eq:lesxU} 
		that $IH^j(\phi)$ is an isomorphism for all $j\in\Z$ and therefore $\phi$ is an isomorphism 
		by loc.\ cit.
		
		\spa
		(iv) The proof for $ \rsect_xF$ is the same and we do not repeat it. 
	\end{proof}
	
	\begin{proposition}[see~\cite{KS90}*{Prop.~3.4.3}]\label{pro:dual}
		Let $F\in \Derb_\wcc(\cor_X)$. Then $\RD_XF\in \Derb_\wcc(\cor_X)$. Moreover, one has the isomorphisms 
		$\rsect_x\RD_XF\simeq \RD(F_x)$ and $(\RD_XF)_x\simeq \RD( \rsect_xF)$. 
	\end{proposition}
	The case of $\Derb_\cc(\cor_X)$ is treated in~\cite{KS90}*{Prop.~3.4.3}. When reading carefully the proof, one 
	checks that the hypotheses that $F_x$ and $\rsect_xF$ are perfect are not used in the proof of the statement of Proposition~\ref{pro:dual}. Hence, the proof in loc.~cit.~indeed already proves this more general statement and we shall not repeat the argument here. 
		\begin{proposition}\label{pro:homtensL}
		Let $F\in\Derb_\wcc(\cor_X)$,  let $L\in\Derb(\cor)$ and let $x\in X$. Then $F\ltens L_X$ and 
		$\rhom(L_X,F)$ belong to  $\Derb_\wcc(\cor_X)$. Moreover, one has the isomorphisms
		\eqn
		&&\rhom(L_X,F)_x\isoto \RHom(L,F_x),\quad (\rsect_xF)\ltens L\isoto \rsect_x(F\tens L_X),
		\eneqn
		\end{proposition}
	\begin{proof}
		(i) One has 
		\eqn
		\RHom(L,F_x)&\simeq&\RHom(L,\sinddlim[x\in U]\rsect(U;F))\simeq\sinddlim[x\in U]\RHom(L, \rsect(U;F))\\
		&\simeq& \sinddlim[x\in U]\rsect(U;\rhom(L_X,F)).
		\eneqn
		The last isomorphism follows from $\RHom(L, \rsect(U;F))\simeq \rsect(U;\rhom(L_X,F))$, which is true for any $F\in \Derb(\cor_X)$ (as recalled later in \eqref{eq:secthom}).
		
		This proves that $\sinddlim[x\in U]\rsect(U;\rhom(L_X,F))$ is representable as well as the first isomorphism.
		
		\spa
		(ii) One has 
		\eqn
		(\rsect_xF)\ltens L&\simeq&(\proolim[x\in U]\rsect_c(U;F))\ltens L\simeq \proolim[x\in U](\rsect_c(U;F)\ltens L)\\
		&\simeq&\proolim[x\in U]\rsect_c(U;F\ltens L_X).
		\eneqn
		 The last isomorphism follows from $\rsect_c(U;F)\ltens L\simeq \rsect_c(U;F\ltens L_X)$, which is true for any $F\in \Derb(\cor_X)$ (as recalled later in \eqref{eq:secthom}).
		
		This proves that $\sproolim[x\in U]\rsect_c(U;F\tens L_X)$ is representable as well as the second isomorphism.
	\end{proof}
	
		\section{Weakly $\R$-constructible sheaves}
	
	The property of being  weakly cohomologically constructible is not stable by the six operations. That is why we shall consider weakly $\R$-constructible sheaves instead.
	Hence, from now on, all manifolds and morphisms of manifolds will be real analytic.

	Let $X$ be  a real analytic manifold. As already mentioned,  we denote by $\Derb_\wRc(\cor_X)$  (resp.\ $\Derb_\Rc(\cor_X)$) the full triangulated subcategory of $\Derb(\cor_X)$ consisting of weakly  $\R$-constructible (resp.\ $\R$-constructible) complexes on $X$. 
	
	Recall that $F\in\Derb(\cor_X)$ belongs to $\Derb_\wRc(\cor_X)$ if and only if its micro-support $\SSi(F)$ is 
	contained in a conic subanalytic isotropic subset of $T^*X$ and this is equivalent to the fact that $\SSi(F)$ is a conic subanalytic Lagrangian subset.
	
	If $F_1,F_2\in \Derb_\wRc(\cor_X)$, then $F_1\ltens F_2$ and $\rhom(F_1,F_2)$ belong to $ \Derb_\wRc(\cor_X)$. Moreover, 
	if $f\cl X\to Y$ is a morphism of real analytic manifolds and $G\in  \Derb_\wRc(\cor_Y)$, then $\opb{f}G$ and $\epb{f}G$  belong to $ \Derb_\wRc(\cor_X)$. If $F\in  \Derb_\wRc(\cor_X)$ and $f$ is proper on $\supp(F)$, then $\reim{f}F\isoto\roim{f}F$ 
	belongs to $ \Derb_\wRc(\cor_Y)$. This follows from~\cite[Prop.~8.4.6]{KS90}.

	Finally, recall (see \cite[\S8.4]{KS90}) that $F\in\Derb_\wRc(\cor_X)$ is $\R$-constructible if for any $x\in X$, $F_x$ is perfect.
	
	The following lemma about the relation between weak $\R$-constructibility and weak cohomological constructibility is well-known. Namely, it follows from~\cite{KS90}*{Lemma~8.4.7} as in the proof of Prop.~8.4.9 in loc.~cit.
	
	\begin{lemma}
		Weakly $\R$-constructible sheaves are weakly cohomologically constructible. In other words, 
		the category 
		$\Derb_\wRc(\cor_X)$ is a full triangulated subcategory of $\Derb_\wcc(\cor_X)$.
	\end{lemma}

	The following proposition will be an important ingredient in many of the isomorphisms we prove below. We give a direct proof for weakly $\mathbb{R}$-constructible sheaves, but we also refer to \cite{Su03}*{Cor.~2.2.6} for more general results in this direction.
	
	\begin{proposition}\label{pro:dualconserv}
		 The functor $\prod_{x\in X}\rsect_x(\scbul)\cl\Derb_{\wRc}(\cor_X)\to\Derb(\cor)$ is conservative. 
	\end{proposition}
	\begin{proof}
Let $F\in \Derb_{\wRc}(\cor_X)$. We need to show that $\rsect_{x}F\simeq 0$ for all $x\in X$ implies $F\simeq 0$.

\spa
	(i) First, assume that $F$ is locally constant. Since the problem is local, we may assume that $F$ is constant, that is, $F=L_X$ for some $L\in\Derb(\cor)$.
 Let $d\geq 1$ be the dimension of $X$. (The statement is clear for $\dim X=0$.) Consider an open ball (in a local chart) $B_{x_0}(\epsilon)$ centred at $x_0$ with radius $\epsilon>0$. 
 Then we consider the distinguished triangle
\eqn
&&\rsect_{x_0}F\To \rsect_{B_{x_0}(\epsilon)}F\To \rsect_{B_{x_0}(\epsilon)\setminus\{x_0\}}F\overset{+1}{\To}.
\eneqn
	After taking global sections, the second and third objects in this triangle become $\rsect(B_{x_0}(\epsilon);L_X)\simeq L$ and $\rsect(B_{x_0}(\epsilon)\setminus \{x_0\};L_X)\simeq L\oplus L[-(d-1)]$, respectively. (Note that $B_{x_0}(\epsilon)$ is contractible and $B_{x_0}(\epsilon)\setminus \{x_0\}$ is homotopy equivalent to a sphere $S^{d-1}$, the boundary of  $B_{x_0}(\epsilon)$.) Therefore, the hypothesis $\rsect_{x_0}L_X\simeq0$ implies $L\simeq L\oplus L[-(d-1)]$, hence $L\simeq0$.
	
\spa
(ii) In the general case, consider a subanalytic stratification $X=\bigsqcup_\alpha X_\alpha$ such that $F$ is locally constant on the strata. Let $d$ be the dimension of $X$ and let us argue by induction on $d$, assuming the result is proved for manifolds of dimension $d-1$. By (i), $F\simeq0$ on the open strata. Let $Z$ be a stratum of 
maximal dimension on which $F\vert_Z$ is (possibly) not zero. 
Denote by $j_Z\cl Z\into X$ the embedding  and let   $x\in Z$. Since $F$ is supported by $Z$ (in a neighborhood of $Z$), $F\simeq\oim{j_Z}\opb{j_Z}F$. 
	Then 
	\eqn
	\rsect_xF&\simeq&\rhom(\cor_{Xx},F)\simeq \rhom(\cor_{Xx},\oim{j_Z}\opb{j_Z}F)\\
	&\simeq&\oim{j_Z}\rhom(\opb{j_Z}\cor_{Xx},\opb{j_Z}F)\simeq \oim{j_Z}\rsect_x(F\vert_Z).
	\eneqn
	If $\rsect_xF\simeq0$, we get $\rsect_x(F\vert_Z)\simeq0$, hence $F\simeq 0$ by the induction hypothesis. 
	\end{proof}

	\subsection*{Inverse images}
	
	\begin{theorem}\label{th:invim}
		Let $f\cl X\to Y$ be a morphism of real analytic manifolds. Let $M,G\in \Derb_\wRc(\cor_Y)$ and assume that $M$ is locally constant. Then 
		\banum
		\item
		$\opb{f}\rhom(M,G)\isoto \rhom(\opb{f}M,\opb{f}G)$,
		\item
		$\epb{f}G\ltens \opb{f}M\isoto \epb{f}(G\ltens M)$.
		\eanum
	\end{theorem}
	\begin{proof}
		Since the problem is local on $Y$, we may assume that $M=L_Y$ is the constant sheaf associated with some $L\in\Derb(\cor)$. Hence $\opb{f}L_Y\simeq L_X$. 
		
		\spa
		(a) Let $x\in X$, and set $y=f(x)$. Applying Proposition~\ref{pro:homtensL}, one gets
		\eqn
		\rhom(L_X,\opb{f}G)_x&\simeq&\rhom(L,(\opb{f}G)_x)\\
		&\simeq&\rhom(L,G_y)\\ &\simeq &\rhom(L_Y,G)_y\simeq (\opb{f}\rhom(L_Y,G))_x.
		\eneqn
		(b) Remark first that for any sheaf $H$ on $Y$, one has $\rsect_x(\epb{f}H)\simeq  \rsect_yH$. 
		Then using Proposition~\ref{pro:homtensL}, one has 
		\eqn
		\rsect_x( \epb{f}G\ltens L_X)&\simeq&(\rsect_x\epb{f}G)\ltens L\simeq (\rsect_yG)\ltens L,\\
		\rsect_x\epb{f}(G\ltens L_Y)&\simeq&\rsect_y(G\ltens L_Y)\simeq (\rsect_yG)\ltens L.
		\eneqn
		Set $A=\epb{f}G\ltens L_X$ and $B=\epb{f}(G\ltens L_Y)$. We have proved that the morphism $A\to B$ induces 
		for  all $x\in X$
		an isomorphism $\rsect_xA\simeq \rsect_xB$. Then $A\simeq B$ by Proposition~\ref{pro:dualconserv}.
	\end{proof}
	
	\subsection*{Tensor product and hom}
	We shall make use of the following well-known result. We nonetheless give a proof for the reader's convenience.
	
	\begin{lemma}\label{lem:noether}
		Let $L,M\in\Derb(\cor)$ and let $N\in\Derb_f(\cor)$. Then 
		\eqn
		&&\RHom(L,M)\ltens N\isoto \RHom(L,M\ltens N).
		\eneqn
	\end{lemma}
	\begin{proof}
		By hypothesis, we may represent $N$ by a bounded complex of projective modules of finite rank. 
		
		\spa
		(i) Assume first that $N=P$ is concentrated in a single degree. 
		If $P$ is  of finite rank, there exists an integer $n$ and an epimorphism $\cor^n\epito P$. If moreover $P$ is projective, then this epimorphism has a retract and we get $\cor^n\simeq P\oplus Q$. This proves the result in this case.
		
		\spa
		(ii) Now assume that $N$ is represented by the complex $0\to P^0\to\cdots\to P^m\to0$ with all $P^j$'s projective of finite rank. 
		Here we assume for simplicity in the notations that $P^0$ is in degree $0$. 
		Assume that the result is proved for complexes of amplitude $\leq m$.
		Let us use the so-called ``stupid truncation''. Denote by $N_0$ the complex $0\to P^0\to\cdots\to P^{m-1}\to0$ and by $u\cl N\to N_0$ the natural morphism. 
		We have an exact sequence of complexes $0\to P^m[-m]\to N\to[u] N_0\to 0$ and it follows that the triangle 
		$P^m[-m]\to N\to[u] N_0\to[+1]$ is distinguished. Arguing by induction on $m$, the proof is complete.
	\end{proof}

	Let $X$ and $Y$ be real analytic manifolds. As usual, one denotes by $q_1$ and $q_2$ the projections from $X\times Y$ to $X$ and $Y$, respectively. 
	One denotes by $\delta\cl X\to X\times X$ the diagonal morphism. One denotes  by $\letens$ the external tensor product
	\eqn
	&&F\letens G\eqdot \opb{q_1}F\ltens\opb{q_2}G.
	\eneqn
	Recall~\cite[Prop.~3.4.4]{KS90} that for $F\in\Derb_\Rc(\cor_X)$ and $G\in\Derb(\cor_Y)$, one has the isomorphism
	\eq\label{eq:Dhom}
	&&\RD_XF\letens G\isoto \rhom(\opb{q_1}F,\epb{q_2}G).
	\eneq
	Also note that we have the following isomorphism for $F_1,F_2\in\Derb(\cor_X)$ and $G_1,G_2\in\Derb(\cor_Y)$:
	\eq\label{eq:etenstens}
	&&(F_1\letens F_2)\ltens(G_1\letens G_2)\simeq (F_1\ltens G_1)\letens(F_2\ltens G_2).
	\eneq
	
	\begin{theorem}\label{th:tenshom}
		Let $K,F_1\in\Derb_\wRc(\cor_X)$, with $K$ locally constant and let $F_2\in\Derb_\Rc(\cor_X)$. 
		Then 
		\banum
		\item 
		$\rhom(K,F_1)\ltens F_2\isoto \rhom(K,F_1\ltens F_2)$.
		\item
		$\rhom(F_2,F_1)\ltens K\isoto \rhom(F_2,F_1\ltens K)$.
		\eanum
	\end{theorem}
	\begin{proof}
		We may assume that $K=L_X$ is the constant sheaf associated with $L\in\Derb(\cor)$. The fact that 
		$F_1\ltens F_2$ and $\rhom(L_X,F_1\ltens F_2)$ belong to $\Derb_\wRc(\cor_X)$ follows from~\cite[Prop.~8.4.6]{KS90}.
		
		\spa
		(a)  Let $x\in X$. One has
		\eqn
		(\rhom(L_X,F_1)\ltens F_2)_x&\simeq&\rhom(L_X,F_1)_x\ltens (F_2)_x\simeq\RHom(L,(F_1)_x)\ltens (F_2)_x\\
		&\simeq&\RHom(L,(F_1)_x\ltens (F_2)_x)\simeq(\rhom(L_X,F_1\ltens F_2))_x.
		\eneqn
		The second and fourth isomorphisms follow from  Proposition~\ref{pro:homtensL} and 
		the third one from Lemma~\ref{lem:noether}.
		
		\spa
		(b) One has 
		\eqn
		\rhom(F_2,F_1)\ltens L_X&\simeq&\epb{\delta}(\RD_XF_2\letens F_1)\ltens \opb{\delta}(\cor_X\letens L_X)\\
		&\simeq&\epb{\delta}((\RD_XF_2\letens F_1)\ltens(\cor_X\letens L_X))\\
		&\simeq&\epb{\delta}((\RD_XF_2\ltens\cor_X)\letens(F_1\ltens L_X))
		\simeq\rhom(F_2,F_1\ltens L_X).
		\eneqn
		Here, the second isomorphism follows from Theorem~\ref{th:invim}~(b). The other ones follow from~\eqref{eq:Dhom} and~\eqref{eq:etenstens}.
	\end{proof}

	\subsection*{Direct images}
	Let $f\cl X\to Y$ be a morphism of real analytic manifolds. Let $F\in \Derb_\wRc(\cor_X)$, and let $M\in\Derb(\cor_Y)$ 
	be  locally constant. One can ask if the morphism
	\eqn
	&&\roim{f}F\ltens M\to \roim{f}(F\ltens\opb{f}M)
	\eneqn
	is an isomorphism. 
	The answer is negative in general, even if we require $F$ to be constructible, thanks to an example of~\cite[Rem.~4.4]{Ho23}.
	
	However, there is a positive answer when considering sheaves {\em constructible  up to infinity}. 
	Before proving the result for general direct images, let us establish it in the particular case of open embeddings.
	
	\begin{lemma}\label{lem:opendirim}
		Let $j\cl U\into X$ be the open embedding of a subanalytic relatively compact open subset $U$ of $X$. 
		Let $F\in \Derb_\wRc(\cor_U)$ and assume that there exists $G\in \Derb_\wRc(\cor_X)$ with $\opb{j}G\simeq F$. 
		Let $K\in\Derb(\cor_X)$ be locally constant. Then
		\banum
		\item
		$\reim{j}\rhom(\opb{j}K,F)\isoto \rhom(K,\reim{j}F)$.
		\item
		$\roim{j}F\ltens K\isoto \roim{j}(F\ltens\opb{j}K)$. 
		\eanum
	\end{lemma}
	\begin{proof}
		As above, we may assume that $K=L_X$ is the constant sheaf associated with $L\in\Derb(\cor)$. 
		Let $G\in\Derb_\wRc(\cor_X)$ be such that $\opb{j}G\simeq F$.
		Then $\roim{j}F\simeq\rsect_UG\simeq  \rhom(\cor_{XU},G)$ and $\reim{j}F\simeq G_U\simeq \cor_{XU}\tens G$.
		
		\spa
		(a)  Note that $\opb{j}\rhom(L_X,G)\simeq \rhom(\opb{j}L_X,F)$. Using Theorem~\ref{th:tenshom} (a), we get
		\eqn
		\reim{j}\rhom(\opb{j}L_X,F)&\simeq&\cor_{XU}\tens\rhom(L_X,G)\simeq\rhom(L_X,G\tens\cor_{XU})\\
		&\simeq&\rhom(L_X,\reim{j}F).
		\eneqn
		
		\spa
		(b) Using Theorem~\ref{th:tenshom}~(b), we have
		\eqn
		\roim{j}F\ltens L_X&\simeq& \rhom(\cor_{XU},G)\ltens L_X\simeq \rhom(\cor_{XU},G\ltens L_X)\\
		&\simeq& \roim{j}\opb{j}(G\tens L_X)   \simeq\roim{j}(F\ltens\opb{j} L_X).
		\eneqn
	\end{proof}
	
	Recall the following notions extracted  from~\cite{Sc23}.
	\begin{definition}
		A b-analytic manifold  $\fX$  is a  pair $(X,\bX)$ with $X\subset \bX$ an open embedding of real analytic manifolds such that $X$ is relatively compact and subanalytic in $\bX$.
		One writes $j_X\colon X\hookrightarrow \widehat{X}$ for the inclusion.
		
		A morphism $f\cl \fX=(X,\bX) \to \fY=(Y,\bY)$ of b-analytic manifolds is a morphism of real analytic manifolds  $f\colon X\to Y$ such that
		the graph $\Gamma_f$ of $f$ in $X\times Y$    is subanalytic in $\bX\times\bY$.
	\end{definition}
	Let $F\in\Derb_\wRc(\cor_X)$. One says that $F$ is \emph{weakly constructible up to infinity} or simply \emph{weakly b-constructible} if 
	$\eim{j_X}F$ (or, equivalently, $\roim{j_X}F$) belongs to $\Derb_\wRc(\cor_\bX)$. One denotes by $\Derb_\wRc(\cor_\fX)$ the full triangulated subcategory of 
	$\Derb_\wRc(\cor_X)$ consisting of weakly b-constructible objects.

	\begin{theorem}\label{th:dirim}
		Let $f\cl\fX\to\fY$ be a morphism of b-analytic manifolds. Let $F\in\Derb_\wRc(\cor_\fX)$ and let $M\in\Derb(\cor_Y)$ be locally constant.
		Then $\reim{f}F$ and $\roim{f}F$ belong to $\Derb_\wRc(\cor_\fX)$ and 
		\banum
		\item
		$\reim{f}\rhom(\opb{f}M,F)\isoto \rhom(M,\reim{f}F)$,
		\item
		$(\roim{f}F)\ltens M\isoto \roim{f}(F\ltens\opb{f}M)$.
		\eanum
	\end{theorem}
	\begin{proof}
	Since the problem is local on $Y$, 
		we may assume that  $M=L_Y$  is the constant sheaf associated with some $L\in\Derb(\cor)$. 
		Set for short $Z\vcentcolon= \bX\times \bY$ and denote by $q_1$ and $q_2$ the first and second projection from $Z$ to $\bX$ and $\bY$, respectively. Denote by $\Gamma_f\subset Z$ the graph of $f$. Note that $\Gamma_f$ is subanalytic in $Z$ by definition, and relatively compact in $Z$ since it is contained in the relatively compact subset $X\times Y$.
		
		\spa
		One has
		\eqn
		\reim{f}F&\simeq&\opb{j_Y}\reim{q_2}(\opb{q_1}\eim{j_X}F\ltens\cor_{\Gamma_f}),\\
		\roim{f}F&\simeq&\epb{j_Y}\roim{q_2}\rhom(\cor_{\Gamma_f},\epb{q_1}\roim{j_X}F).
		\eneqn
		Note that the supports of $\opb{q_1}\eim{j_X}F\tens\cor_{\Gamma_f}$ and $\rhom(\cor_{\Gamma_f}, \epb{q_1}\roim{j_X}F)$
		are contained in $\ol\Gamma_f$ and hence compact in $Z$. 
		
		\spa
		(a) Let us apply the functor  $\rhom(L_Y,\scbul)$ to the first isomorphism. We get
		\eqn
		\rhom(L_Y,\reim{f}F)&\simeq& \rhom(\opb{j_Y}L_\bY,\opb{j_Y}\roim{q_2}(\opb{q_1}\eim{j_X}F\ltens\cor_{\Gamma_f}))\\
		&\simeq& \opb{j_Y}\rhom(L_\bY,\roim{q_2}(\opb{q_1}\eim{j_X}F\ltens\cor_{\Gamma_f}))\\
		&\simeq& \opb{j_Y}\roim{q_2}\rhom(L_{\bX\times\bY},\opb{q_1}\eim{j_X}F\ltens\cor_{\Gamma_f})\\
		&\simeq& \opb{j_Y}\roim{q_2}(\rhom(\opb{q_1}L_\bX,\opb{q_1}\eim{j_X}F)\ltens\cor_{\Gamma_f})\\
		&\simeq& \opb{j_Y}\roim{q_2}(\opb{q_1}\rhom(L_\bX,\eim{j_X}F)\ltens\cor_{\Gamma_f})\\
		&\simeq& \opb{j_Y}\roim{q_2}(\opb{q_1}\eim{j_X}\rhom(L_X,F)\ltens\cor_{\Gamma_f})\\
		&\simeq& \reim{f} \rhom(f^{-1}L_Y,F).
		\eneqn
		The second and fifth isomorphism use Theorem~\ref{th:invim} (a), the fourth uses Theorem~\ref{th:tenshom} (a), and the sixth isomorphism uses Lemma~\ref{lem:opendirim} (a). On the other hand, the third isomorphim is classical (see \cite[(2.6.15)]{KS90}).
		
		\spa
		(b) The proof is completely analogous, using parts (b) of the statements mentioned above instead.
	\end{proof}
	Recall the following isomorphisms which hold for any $F\in \Derb(\cor_X)$, $L\in\Derb(\cor)$ and any open $U\subset X$. One has
	\eq\label{eq:secthom}
	&&\ba{l}\rsect_c(U;F\ltens L_X)\isoto \rsect_c(U;F)\ltens L,\\
	\rsect(U;\rhom(L_X,F))\simeq \RHom(L,\rsect(U;F)).
	\ea\eneq	
	
	\begin{corollary}
		Let $F\in\Derb_\wRc(\cor_X)$ and let $L\in\Derb(\cor)$. Let $U$ be an open relatively compact subanalytic subset of $X$. Then 
		\banum
		\item
		$\rsect_c(U;\rhom(L_X,F))\isoto \RHom(L,\rsect_c(U;F))$,
		\item
		$\rsect(U,F\ltens L_X)\simeq \rsect(U;F)\ltens L$.
		\eanum
	\end{corollary}
	\begin{proof}
		Apply Theorem~\ref{th:dirim} to the sheaf $F|_U$ with $\fX=(U,X)$, $\fY=(\rmpt,\rmpt)$ and $f=a_U$. 
	\end{proof}
	
In particular, when $X_\infty=(X,\bX)$ is b-analytic and $F\in\Derb_\wRc(\cor_\fX)$, one obtains
	\eqn
	&&\rsect_c(X;\rhom(L_X,F))\isoto \RHom(L,\rsect_c(X;F)),\\
	&&\rsect(X;F\ltens L_X)\simeq \rsect(X;F)\ltens L.
	\eneqn
	
	\subsection*{Duality}
	From our above result, we obtain slight generalisations of well-known statements about the behaviour of the duality functor with respect to direct and inverse images.
	
	Recall that for any $G\in\Derb(\cor_Y)$ and a continuous $f\colon X\to Y$, one has $\RD_X\opb{f}G\simeq\epb{f}\RD_YG$. In general, the isomorphism $\RD_X\epb{f}G\simeq\opb{f}\RD_YG$ does not hold, but it holds if $f$ is a morphism of real analytic manifolds and $G\in \Derb_\Rc(\cor_Y)$ (cf.~\cite[Exe.~VIII.3]{KS90} (ii)). The following corollary generalises this statement.
	
	\begin{corollary}
		Let $f\cl X\to Y$ be a morphism of real analytic manifolds. Let $G\in \Derb_\Rc(\cor_Y)$ and assume that $M\in\Derb(\cor_Y)$ is locally constant. 
		Then
		\eqn
		&&\RD_X\epb{f}(G\ltens M)\simeq \opb{f}\RD_Y(G\ltens M).
		\eneqn
	\end{corollary}
	\begin{proof}
		One has the chain of isomorphisms
		\eqn
		\RD_X\epb{f}(G\ltens M)&\simeq&\rhom(\epb{f}G\ltens \opb{f}M,\omega_X)\\
		&\simeq&\rhom(\opb{f}M, \RD_X(\epb{f}G))\simeq \rhom(\opb{f}M, \opb{f}\RD_YG)\\
		&\simeq&\opb{f}\rhom(M, \RD_YG)\simeq \opb{f}\RD_Y(G\ltens M).
		\eneqn
	Here, we have used Theorem~\ref{th:invim} (b) in the first isomorphism. Note also that the third isomorphism uses the classical fact (as mentioned above) that $\RD_X\epb{f}G\simeq\opb{f}\RD_YG$ for $G\in \Derb_\Rc(\cor_Y)$.
	\end{proof}
	
	Recall that for any $F\in\Derb(\cor_X)$, one has $\RD_Y\reim{f}F\simeq\roim{f}\RD_XF$. On the contrary, the isomorphism $\reim{f}\RD_XF\simeq \RD_Y\roim{f}F$ does not hold in general. However, it holds under the assumption that $f$ is a morphism of real analytic manifolds and that $F$ and $\reim{f}\RD_X F$ are $\mathbb{R}$-constructible complexes (cf.~\cite[Exe.~VIII.3]{KS90} (iii)). In particular, this means that it holds if $F\in \Derb_\Rc(\cor_\fX)$, since the six functors preserve $\mathbb{R}$-constructibility in the b-analytic setting (see \cite[Cor.~2.13]{Sc23}). The following corollary is a generalisation of the latter statement.
	
	\begin{corollary}
		Let $f\cl\fX\to\fY$ be a morphism of b-analytic manifolds. Let $F\in\Derb_\Rc(\cor_\fX)$, and let $M\in\Derb(\cor_Y)$ be locally constant. Then
		\eqn
		&&\reim{f}\RD_X(F\ltens\opb{f}M)\simeq \RD_Y(\roim{f}F\ltens M).
		\eneqn
	\end{corollary}
	\begin{proof}
		One has the chain of isomorphisms
		\eqn
		\reim{f}\RD_X(F\ltens\opb{f}M)&\simeq &\reim{f}\rhom(\opb{f}M,\RD_XF)\\
		&\simeq& \rhom(M,\reim{f}\RD_XF)\simeq \rhom(M,\RD_Y\roim{f}F)\\
		&\simeq &\RD_Y(\roim{f}F\ltens M).
		\eneqn
	Here, we have used Theorem~\ref{th:dirim} (a) in the second isomorphism. Note that in the third isomorphism we have used the fact that $\reim{f}\RD_XF\simeq \RD_Y\roim{f}F$ for $F\in\Derb_\Rc(\cor_\fX)$, as mentioned above.
	\end{proof}
		
	\subsection*{Field extensions and micro-support}
Consider the case where $\cor$ is a field and $\corex$ is a field extension of $\cor$. Then, of course, all preceding results apply in particular with $K=\corex_X$ and $M=\corex_Y$,  i.e.\ taking the constant sheaf with stalk $\corex$ as the locally constant sheaves in the above statements. This is the case of interest in \cite{Ho23}, where extension of scalars for sheaves of vector spaces is studied.
		Our above results apply, however, to more general situations, since we put weaker restrictions on the objects involved, as described in the introduction.

		Recall  that \cite[Rem.~5.1.5]{KS90} asserts that if 
		$\for$ denotes the forgetful functor
		\eqn
		&&\for\cl\Derb(\corex_X)\to\Derb(\cor_X),
		\eneqn
		and if $F\in\Derb(\corex_X)$, then the micro-support of $F$ and that of $\for(F)$ are the same.
 However, this result does not mean that for $F\in\Derb(\cor_X)$, one has $\SSi(F)=\SSi(F\tens[\cor_X]\corex_X)$. We shall discuss this last point now.
	\medskip
		
	Remark first that for $F,K\in\Derb(\cor_X)$ with $K$ locally constant, one has
	\eqn
	&&\SSi(K\ltens F)\subset\SSi(F),\quad \SSi(\rhom(K,F))\subset\SSi(F).
	\eneqn
	Indeed, one has by~\cite[Prop.~5.4.14]{KS90}
	\eqn
	&&\SSi(K\ltens F)\subset T^*_XX+\SSi(F)=\SSi(F),
	\eneqn
	and similarly with $\rhom(K,F)$.
	
Finally, let us state the following useful result, well-known to specialists,  which says that the converse inclusion is true over a field.

\begin{proposition}
Assume that $\cor$ is a field.  
Let $F,K\in\Derb(\cor_X)$ with $K$  locally constant, $X$ connected and $K\neq0$. Then
\eqn
		&&\SSi(K\tens F)=\SSi(F),\quad  \SSi(\rhom(K,F))=\SSi(F).
		\eneqn
	\end{proposition}
	\begin{proof} 
		The problem is local and we may assume that $K=L_X$ is the constant sheaf associated with $L\in\Derb(\cor)$. Then $L\simeq\oplus_j H^j(L)\,[-j]$.
		Since the micro-support of a finite direct sum is the union of the summands' micro-supports (which is a direct consequence of the definition of the micro-support \cite[Def.~5.1.2]{KS90}), we may assume that $L\in\md[\cor]$. In this case, there exists $L'\in\md[\cor]$ such that $L\simeq\cor\oplus L'$  (since every vector space has a basis), hence
		$L_X\simeq \cor_X\oplus L'_X$. Therefore, $L_X\tens F\simeq F\oplus G$, with $G=L'_X\tens F$. Since again 
		$\SSi(F\oplus G)=\SSi(F)\cup\SSi(G)$,  the result for $L_X\tens F$ follows. The case of $\rhom(L_X,F)$ is similar.  
	\end{proof}

\providecommand{\bysame}{\stLeavevmode\hbox to3em{\hrulefill}\thinspace}

\vspace*{1cm}
\noindent
\begin{tabular}{cc}
\parbox[t]{16em}
{\scriptsize{
		Andreas~Hohl\\
		Université Paris Cité and Sorbonne Université,\\
		CNRS, IMJ-PRG, F-75013 Paris, France\vspace{0.1cm}\\
		\textit{Current address:}\\
		Fakultät für Mathematik, Technische Universität Chemnitz, 09107 Chemnitz, Germany\\
		e-mail: andreas.hohl@math.tu-chemnitz.de\\
		https://www.andreashohl.eu}}
		&
\parbox[t]{16em}{\scriptsize{
		Pierre Schapira\\
		Sorbonne Universit{\'e},\\
		CNRS IMJ-PRG Paris, France\\\
		e-mail: pierre.schapira@imj-prg.fr\\
		http://webusers.imj-prg.fr/\textasciitilde pierre.schapira/
}}
\end{tabular}


\begin{thebibliography}{15}
	
	\bibitem[Hoh23]{Ho23} A. Hohl, \emph{An introduction to field extensions and Galois descent for sheaves of vector spaces}, Preprint (2023), \texttt{arxiv:2302.14837v2}
	
	\bibitem[KS90]{KS90} M. Kashiwara and P. Schapira, \emph{Sheaves on Manifolds}, Grundlehren Math. Wiss., vol. 292, Springer-Verlag, Berlin 1990.
	
	\bibitem[KS06]{KS06} M. Kashiwara and P. Schapira, \emph{Categories and Sheaves}, Grundlehren Math. Wiss., vol. 332, Springer-Verlag, Berlin 2006.
	
	\bibitem[Sch23]{Sc23} P. Schapira, \emph{Constructible sheaves and functions up to infinity}, J. Appl. Comput. Topol. \textbf{7} (2023), 707--739.
	
	\bibitem[Sch03]{Su03} J. Schürmann, \emph{Topology of Singular Spaces and Constructible Sheaves}, IMPAN Monogr. Mat. (N. S.), vol. 63, Birkhäuser, Basel 2003.

\end{thebibliography}
\end{document}